\documentclass{article}
\usepackage{amssymb,amsthm,amsmath,enumerate}
\usepackage{graphics}
\usepackage{mathabx}

\theoremstyle{plain}

\newtheorem{theorem}{Theorem}
\newtheorem{lemma}{Lemma}
\newtheorem{corollary}{Corollary}
\newcommand*{\supp}{\ensuremath{\mathrm{supp\,}}}
\newcommand*{\Ann}{\ensuremath{\mathrm{Ann\,}}}
\newcommand*{\Aut}{\ensuremath{\mathrm{Aut\,}}}
\newcommand*{\Ker}{\ensuremath{\mathrm{Ker\,}}}
\newcommand*{\R}{\ensuremath{\mathbb{R}}}

\newcommand*{\C}{\ensuremath{\mathbb{C}}}
\newcommand*{\N}{\ensuremath{\mathbb{N}}}

\begin{document}

\date{}

\author{{L\'aszl\'o Sz\'ekelyhidi}\\
   {\small\it Institute of Mathematics, University of Debrecen,}\\
   {\small\rm e-mail: \tt lszekelyhidi@gmail.com} }

\title{Spherical spectral synthesis
   \footnotetext{The research was partly supported by the
   Hungarian National Foundation for Scientific Research (OTKA),
   Grant No. K111651.}\footnotetext{AMS (2000)
   Subject Classification: 43A45, 43A90, 22D15}}

\maketitle

\begin{abstract}
In this paper we make an attempt to extend L.~Schwartz's classical result on spectral synthesis to several dimensions. Due to counterexamples of D.~I.~Gurevich this is impossible for translation invariant varieties. Our idea is to replace translations by proper euclidean motions in higher dimensions. For this purpose we introduce the basic concepts of spectral analysis and synthesis in the non-commutative setting based on Gelfand pairs, where "translation invariance" will be replaced by invariance with respect to a compact group of automorphisms. The role of exponential functions will be played by spherical functions. As an application we obtain the extension of L. Schwartz's fundamental result.
\end{abstract}

\section{Notation and terminology}
\hskip.5cm
In this paper we propose a general setting for spectral analysis and synthesis on non-commutative locally compact groups. It is clear that exponentials and exponential monomials, which serve as basic building blocks for spectral synthesis on commutative groups will not be able to play a similar role on non-commutative groups. Nevertheless, some commutative ideas can be utilized in the case of Gelfand pairs. In our approach exponentials will be replaced by spherical functions and translation invariance will be replaced by invariance with respect to a given compact subgroup, in particular, to a compact group of automorphisms. We apply these ideas in the case of the Gelfand pair $(\R^n, SO(n))$, where the classical concept of "translation invariance" in spectral synthesis will be replaced by invariance under Euclidean motions. It turns out that this "spherical spectral synthesis" can be considered as a generalization of L.~Schwartz's classical spectral synthesis result from the reals to $\R^n$.
\vskip.1cm

Although the basics of the theory of spherical functions and Gelfand pairs can be found in  \cite{MR0621691} (see also \cite{MR2640609}) here we include all necessary concepts and results -- for the sake of completeness.
\vskip.1cm

In this paper $\C$ denotes  the set of complex numbers. For a locally compact group $G$ we denote by $\mathcal C(G)$ the locally
convex topological vector space of all continuous complex valued functions defined on $G$, equipped with the pointwise
operations and with the topology of uniform convergence on compact sets. For each function $f$ in $\mathcal C(G)$ we define $\widecheck{f}$ by $\widecheck{f}(x)=f(x^{-1})$, whenever $x$ is in $G$. 
\vskip.1cm

It is known that the dual of $\mathcal C(G)$ can be identified with the space $\mathcal M_c(G)$ of all compactly supported complex Borel measures on $G$ which is equipped with the pointwise operations and with the weak*-topology. 
The pairing between $\mathcal C(G)$ and $\mathcal M_c(G)$ is given by the formula
$$
\langle \mu,f\rangle=\int f\,d\mu\,.
$$

We shall use the following theorem, describing the dual of $\mathcal M_c(G)$. A more general theorem together with the proof can be found in \cite{KeN76}, 17.6, p.~155. 

\begin{theorem}\label{stardual}
Let $G$ be a locally compact group. For every weak*-continuous linear functional $F:\mathcal M_c(G)\to\C$ there exists a unique continuous function $f:G\to \C$ such that $F(\mu)=\mu(f)$ for each $\mu$ in $\mathcal M_c(G)$.
\end{theorem}

The measure $\widecheck{\mu}$ is defined by 
$$
\langle\widecheck{\mu},f\rangle=\langle\mu,\widecheck{f}\rangle
$$
for each $\mu$ in $\mathcal M_c(G)$ and $f$ in $\mathcal C(G)$.
\vskip.1cm

{\it Convolution} in $\mathcal M_c(G)$ is defined by
\[
\langle \mu*\nu,f\rangle=\int f(x y)\,d\mu(x)\,d\nu(y)
\]
for each $\mu,\nu$ in $\mathcal M_c(G)$ and $x$ in $G$. Convolution converts the linear space $\mathcal M_c(G)$ into a topological algebra with unit $\delta_e$, $e$ being the identity in $G$.  
\vskip.1cm

Convolution of measures in $\mathcal M_c(G)$ with arbitrary functions in $\mathcal C(G)$ is defined  by the similar formula
$$
\mu*f(x)=\int f(y^{-1} x)\,d\mu(y)
$$
for each $\mu$ in $\mathcal M_c(G)$, $f$ in $\mathcal C(G)$ and $x$ in $G$. 
It is clear that, equipped with the action $f\mapsto \mu*f$, the space $\mathcal C(G)$ is a topological left module over $\mathcal M_c(G)$. 
\vskip.1cm

In what follows $K$ will always denote a compact subgroup in $G$ with normed Haar measure $\omega$. We recall that $\omega$ is left invariant, right invariant and inversion invariant. 
\vskip.1cm

The function $f$ in $\mathcal C(G)$ is called {\it bi-$K$-invariant}, or simply {\it $K$-invariant} if  it satisfies
$$
f(k x l)=f(x)
$$
for each $x$ in $G$ and $k,l$ in $K$. All $K$-invariant functions form a closed subspace in the topological vector space $\mathcal C(G)$, which we denote by $\mathcal C(G//K)$. Clearly, $\widecheck{f}$ is $K$-invariant, if $f$ is $K$-invariant.
\vskip.1cm

For each $f$ in $\mathcal C(G)$ the function defined by
$$
f^{\#}(x)=\int_K\,\int_K f(k x l)\,d\omega(k)\,d\omega(l),
$$
for $x$ in $G$ is called the {\it projection} of $f$. Obviously, $f^{\#}$ is $K$-invariant, further the function $f$ in $\mathcal C(G)$ is $K$-invariant if and only if $f^{\#}=f$. We note that if $G$ is Abelian, then the space $\mathcal C(G//K)$ can be identified with $\mathcal C(G/K)$. Clearly, $(f^{\#})\,\widecheck{}=(\widecheck{f})^{\#}$ holds for each $f$ in $\mathcal C(G)$.
\vskip.1cm

A measure $\mu$ in $\mathcal M_c(G)$ is called {\it $K$-invariant}, if
$$
\langle \mu,f\rangle=\langle \mu,f^{\#}\rangle
$$
holds for each $f$ in $\mathcal C(G)$.
We also define the {\it projection} $\mu^{\#}$ of $\mu$ by the equation
$$
\langle \mu^{\#},f\rangle =\int_G \int_K \int_K f(k x l)\, d\omega(k)\, d\omega(l)\, d\mu(x)
$$
for each $f$ in $\mathcal C(G)$. Obviously, $\mu^{\#}$ is in $\mathcal M_c(G)$ and it is $K$-invariant, further $\langle \mu^{\#},f\rangle =\langle \mu,f^{\#}\rangle$. It follows that the measure $\mu$ in $\mathcal M_c(G)$ is $K$-invariant if and only if $\mu^{\#}=\mu$. 
\vskip.1cm

A special role is played by the projections of the evaluation functionals $\delta_y$ defined by
$$
\langle\delta_y, f\rangle=f(y)
$$ 
for each $y$ in $G$ and $f$ in $\mathcal C(G)$, hence
$$
\langle\delta_y^{\#}, f\rangle=f^{\#}(y)=\int_K f(k y l)\,d\omega(k)\,d\omega(l).
$$
Using these measures we define $K$-translation by $y$ in $G$ for each $f$ in $\mathcal C(G)$ as the function $\tau_y f$ defined by the equation
$$
\tau_y f(x)=\delta_{y^{-1}}^{\#}*f(x)=\int_K \int_K f(k y l x)\,d\omega(k)\, d\omega(l)
$$
for each $x$ in $G$. In particular, for each $K$-invariant function $f$ we have
$$
\tau_y f(x)=\int_K f(y k x)\,d\omega(k)
$$
whenever $x,y$ are in $G$. A subset $H$ in $\mathcal C(G//K)$ is called {\it $K$-translation invariant}, if for each $f$ in $H$ and $y$ in $G$ the function $\tau_y f$ is in $H$. A closed $K$-translation invariant linear subspace of $\mathcal C(G//K)$ is called a {\it $K$-variety}. Clearly, the intersection of any  family of $K$-varieties is a $K$-variety. The intersection of all $K$-varieties including the $K$-invariant function $f$ is called the $K$-variety {\it generated by $f$}, or simply the {\it $K$-variety of $f$}, and it is denoted by $\tau(f)$. This is the closure of the linear space spanned by all $K$-translates of $f$.

\section{The dual of $\mathcal C(G//K)$}
\hskip.5cm
The following theorem  describes the space of all continuous linear functionals of the space $\mathcal C(G//K)$, that is, the dual space $\mathcal C(G//K)^{*}$.

\begin{theorem}\label{dualspace}
The dual of $\mathcal C(G//K)$ is identical with the set of the restrictions of all $K$-invariant measures in $\mathcal M_c(G)$ to $\mathcal C(G//K)$. 
\end{theorem}

\begin{proof}
Suppose that $\lambda$ is in $\mathcal C(G//K)^{*}$ and we define
$$
\langle\widetilde{\lambda},f\rangle=\langle\lambda,f^{\#}\rangle
$$
for each $f$ in $\mathcal C(G)$. Then obviously $\widetilde{\lambda}$ is in $\mathcal M_c(G)$. We have
$$
\langle (\widetilde{\lambda})^{\#},f\rangle =\langle \widetilde{\lambda},f^{\#}\rangle =\langle \lambda,f^{\#\#}\rangle =\langle \lambda,f^{\#}\rangle =\langle\widetilde{\lambda},f\rangle
$$
for each $f$ in $\mathcal C(G)$, hence $\widetilde{\lambda}$ is $K$-invariant. If $f$ is in $\mathcal C(G//K)$, then $f=f^{\#}$ and we infer
$$
\langle\widetilde{\lambda},f\rangle=\langle\widetilde{\lambda},f^{\#}\rangle= \langle \lambda,f^{\#\#}\rangle =\langle \lambda,f^{\#}\rangle =\langle \lambda,f\rangle,
$$
hence $\widetilde{\lambda}$ coincides with $\lambda$ on $\mathcal C(G//K)$. This shows that $\lambda$ is the restriction of a $K$-invariant measure in $\mathcal M_c(G)$.
\vskip.1cm

Conversely, it is clear that the restriction of every $K$-invariant measure in $\mathcal M_c(G)$ to $\mathcal C(G//K)$ is a linear functional on $\mathcal C(G//K)$, hence it is an element of $\mathcal C(G//K)^{*}$.
\end{proof}

The dual of $\mathcal C(G//K)$ will be denoted by $\mathcal M_c(G//K)$. It follows from the previous theorem that with the convolution of measures restricted to $\mathcal M_c(G//K)$ the latter is a topological algebra with unit $\delta_e$. Further, $\mathcal C(G//K)$ is a left module over $\mathcal M_c(G//K)$ with the ordinary convolution $f\mapsto \mu*f$. Closed submodules of this module are exactly the $K$-varieties. Left, right and two-sided ideals in $\mathcal M_c(G//K)$ are called left $K$-ideal, right $K$-ideal and $K$-ideal, respectively.

\begin{theorem}\label{dense}
All finitely supported $K$-invariant measures form a dense subalgebra in $\mathcal M_c(G//K)$.
\end{theorem}

\begin{proof}
Let $\mu$ be a $K$-invariant measure and let $(\nu_{\alpha})_{\alpha\in A}$ be a generalized sequence of finitely supported measures in $\mathcal M_c(G)$ converging to $\mu$ in the weak*-topology. Then $(\nu_{\alpha}^{\#})_{\alpha\in A}$ is a generalized sequence of finitely supported measures in $\mathcal M_c(G//K)$ converging to $\mu^{\#}=\mu$ in the weak*-topology.
\end{proof}

\section{Orthogonal complements and annihilators}
\begin{theorem}\label{anni1}
Let $I$ be a left $K$-ideal. Then 
$$
I^{\perp}=\{f:\, f\in \mathcal C(G//K),\,\langle \mu,f\rangle=0\enskip\text{for each}\enskip \mu\in I\}
$$
is a $K$-variety.
\end{theorem}

\begin{proof}
Clearly, $I^{\perp}$ is a closed linear space in $\mathcal C(G//K)$. Let $f$ be in $I^{\perp}$ and $y$ in $G$, then we have for each $\mu$ in $I$
$$
\langle \mu,\delta_{y^{-1}}^{\#}*f\rangle=\int_G \int_K \int_K f(k y l x)\,d\omega(k)\, d\omega(l)\,d\mu(x)=
$$
$$
\int_G \int_G f(u x)\, d\delta_y^{\#}(u)\,d\mu(x)=
\int_G f(t)\,d(\delta_{y}^{\#}*\mu)(t)=
\langle \delta_{y}^{\#}*\mu,f\rangle=0,
$$
as $\delta_{y}^{\#}*\mu$ is in $I$.
\end{proof}

\begin{theorem}\label{anni2}
Let $V$ be a $K$-variety. Then
$$
V^{\perp}=\{\mu:\, \mu\in \mathcal M_c(G//K),\langle \mu,f\rangle=0\enskip\text{for each}\enskip f\in V\}
$$
is a closed left $K$-ideal.
\end{theorem}

\begin{proof}
Obviously, $V^{\perp}$ is a closed linear subspace in $\mathcal M_c(G//K)$. Let $\mu$ in $V^{\perp}$, $\nu$ in $\mathcal M_c(G//K)$, and $f$ in $V$. Then we have to show that
$$
\langle \nu*\mu,f\rangle=\int_G f(x y)\,d\nu(x)\,d\mu(y)=0.
$$
As convolution in $\mathcal M_c(G//K)$ is continuous in both variables, it is enough to prove this for $\nu=\delta_u^{\#}$, by Theorem \ref{dense}, where $u$ is in $G$. On the other hand, we have, as above
$$
\langle \delta_u^{\#}*\mu, f\rangle=\langle \mu, \delta_{u^{-1}}^{\#}* f\rangle=0,
$$
by the $K$-translation invariance of $V$.
\end{proof}

Obviously, the statement remains true if we assume only that $V$ is a $K$-translation invariant set.

\begin{theorem}\label{duality}
For each $K$-variety $V$ we have $V^{\perp\perp}=V$. For each closed left $K$-ideal $I$ we have $I^{\perp\perp}=I$.
\end{theorem}

\begin{proof}
Obviously, we have $V^{\perp\perp}\supseteq V$ and $I^{\perp\perp}\supseteq I$. 
\vskip.1cm

Suppose now that $V^{\perp\perp}\supsetneq V$. Consequently, there is a function $f$ in $V^{\perp\perp}$ such that $f$ is not in $V$. By the Hahn--Banach Theorem, and by Theorem \ref{dualspace}, there is a $\lambda$ in $\mathcal M_c(G//K)$ such that $\langle\lambda,f\rangle\ne 0$, and $\lambda$ vanishes on $V$. This means that $\lambda$ is in $V^{\perp}$, and $f$ is not in $V^{\perp\perp}$, a contradiction. Hence $V^{\perp\perp}=V$.
\vskip.1cm

Similarly, suppose that $I^{\perp\perp}\supsetneq I$ and let $\mu$ be in $I^{\perp\perp}$ such that $\mu$ is not in $I$. As $I$ is closed, by the Hahn--Banach Theorem, there is a linear functional $\Lambda$ in the dual space $\mathcal M_c(G//K)^{*}$ such that $\Lambda$ annihilates $I$ but $\Lambda(\mu)\ne 0$. By Theorem \ref{stardual}, every weak*-continuous linear functional on a dual space arises from an element of the original space, that is, there is an $f_{\Lambda}$ in $\mathcal C(G//K)$ with $\Lambda(\nu)=\langle \nu,f_{\Lambda}\rangle$ holds for each for each $\nu$ in $\mathcal M_c(G//K)$. As $\Lambda$ annihilates $I$ we have 
$$
\Lambda(\nu)=\langle \nu,f_{\Lambda}\rangle=0
$$
for each $\nu$ in $I$, hence $f_{\Lambda}$ is in $I^{\perp}$. On the other hand, 
$$
0\ne \Lambda(\mu)=\langle \mu, f_{\Lambda}\rangle,
$$
a contradiction, as $\mu$ is in $I^{\perp\perp}$. The proof is complete.
\end{proof}

Besides orthogonal complements $V^{\perp}$ and $I^{\perp}$ we consider annihilators as follows. As $\mathcal C(G//K)$ is a topological module over the algebra $\mathcal M_c(G//K)$, the annihilators of subsets in $\mathcal C(G//K)$, resp. in $\mathcal M_c(G//K)$ have the usual meaning from module theory. Let $V$, resp. $I$ be a $K$-variety, resp. a left $K$-ideal. Then the {\it annihilator} of $V$ in $\mathcal M_c(G//K)$, resp. of $I$ in $\mathcal C(G//K)$ is defined as
$$
\Ann V=\{\mu:\, \mu\in \mathcal M_c(G//K)\enskip\text{and}\enskip \mu*f=0\enskip\text{for each}\enskip f\enskip\text{in}\enskip V\},
$$
resp.
$$
\Ann I=\{f:\, f\in \mathcal C(G//K)\enskip\text{and}\enskip \mu*f=0\enskip\text{for each}\enskip \mu\enskip\text{in}\enskip I\}.
$$
Clearly, $\Ann V$, resp. $\Ann I$ are closed subspaces in $\mathcal M_c(G//K)$, resp. in $\mathcal C(G//K)$. We shall use the notation
$$
\widecheck{H}=\{\widecheck{f}:\, f\in H\},\enskip \widecheck{L}=\{\widecheck{\mu}:\, \mu\in L\}.
$$
for each subset $H$ in $\mathcal C(G)$ and $L$ in $\mathcal M_c(G)$.

\begin{lemma} 
For each $K$-variety $V$ and left $K$-ideal $I$ we have
$$
\Ann V=(\widecheck{V})^{\perp},\hskip1cm \Ann I=(\widecheck{I})^{\perp}.
$$
\end{lemma}

\begin{proof}
Let $\mu$ be in $\Ann V$ and $f$ in $V$, then we have 
$$
\langle \mu, \widecheck{f}\rangle =\int_G f(x^{-1})\,d\mu(x)=\mu*f(e)=0,
$$
hence $\mu$ is in $(\widecheck{V})^{\perp}$. Conversely, let $\mu$ be in $(\widecheck{V})^{\perp}$ and $f$ in $V$. Then 
$$
\mu*f(x)=\mu^{\#}*f(x)=\int_G f(y^{-1} x)\,d\mu^{\#}(y)=
$$
$$
\int_G \int_K \int_K f(k y^{-1} l x)\,d\omega(k)\,d\omega(l)\,d\mu(y)= \int_G \int_K f(y^{-1} l x)\,d\omega(l)\,d\mu(y)=
$$
$$
\int_G \int_K \widecheck{f}(x^{-1} l y)\,d\omega(l)\,d\mu(y)=\langle \mu, \delta_x^{\#}*\widecheck{f}\rangle=0
$$
as $\widecheck{V}$ is $K$-variety, hence $\delta_x^{\#}*\widecheck{f}$ is in $\widecheck{V}$. 
\vskip.1cm

To prove the second statement we suppose that $f$ is in $(\widecheck{I})^{\perp}$ and $\mu$ is in $I$. We can proceed as follows:
$$
\mu*f(x)=(\mu*f)\,\widecheck{}\,(x^{-1})=\int_G (\mu*f)\,\widecheck{}\,(t)\,d\delta_{x^{-1}}(t)=
$$
$$
\int_G \bigl[(\mu*f)\,\widecheck{}\,\bigr]^{\#}\,(t)\,d\delta_{x^{-1}}(t)=\int_G (\mu*f)\,\widecheck{}\,(t)\,d\delta_{x^{-1}}^{\#}(t)=
$$
$$
(\mu*f)(t^{-1})\,d\delta_{x^{-1}}^{\#}(t)=\delta_{x^{-1}}^{\#}*(\mu*f)(e)=(\delta_{x^{-1}}^{\#}*\mu)*f(e)=
$$
$$
\langle \delta_{x^{-1}}^{\#}*\mu,\widecheck{f}\rangle=\langle (\delta_{x^{-1}}^{\#}*\mu)\,\widecheck{},f\rangle=0,
$$
as $\delta_{x^{-1}}^{\#}*\mu$ is in $I$. Conversely, if $f$ is in $\Ann I$ and $\mu$ is in $\widecheck{I}$, then $\widecheck{\mu}$ is in $I$ and we infer
$$
0=\widecheck{\mu}*f(e)=\int_G f(t^{-1})\,d\widecheck{\mu}(t)=\langle \mu,f\rangle,
$$
hence $f$ is in $(\widecheck{I})^{\perp}$.
\end{proof}

\begin{corollary}
For each $K$-variety $V$ and closed $K$-ideal $I$ we have
$$
\Ann (\Ann V)=V,\hskip1cm \Ann (\Ann I)=I.
$$
\end{corollary}

\section{Gelfand pairs}
\hskip.5cm
 We call $(G,K)$ a {\it Gelfand pair} if the algebra $\mathcal M_c(G//K)$ is commutative. Obviously, this is the case, for instance, if $G$ is commutative. We note that, although this definition is formally different from the one used in \cite{MR0621691} (see also \cite{MR2640609}) Theorem \ref{app2} and Theorem \ref{approx} below show that the two definitions coincide.
 
 \begin{theorem}\label{commut}
$(G,K)$ is a Gelfand pair if and only if the measures $(\delta_y^{\#})_{y\in G}$ form a commuting family. 
 \end{theorem}
 
\begin{proof}
The necessity is obvious, and the sufficiency follows from Theorem \ref{dense} and from the separate continuity of convolution in $\mathcal M_c(G//K)$.
 \end{proof}

Let $\mathcal C_c(G)$ denote the space of compactly supported continuous complex valued functions on $G$ with the inductive limit topology of the subspaces  of continuous complex valued functions with support in a given compact subset, equipped with the topology of uniform convergence. If $\lambda$ is a fixed left Haar measure on $G$, then with the convolution
$$
\varphi*\psi(x)=\int_G \varphi(t^{-1} x) \psi(t)\,d\lambda(t)
$$ 
$\mathcal C_c(G)$ is a topological algebra. For each $\varphi$ in $\mathcal C_c(G)$ we define
$$
\langle \mu_{\varphi},f\rangle=\int_G f(x) \varphi(x)\,d\lambda(x)
$$
whenever $f$ is in $\mathcal C(G)$. 

\begin{theorem}\label{app1}
The mapping $\varphi\mapsto \mu_{\varphi}$ is a continuous algebra homomorphism of $\mathcal C_c(G)$ into $\mathcal M_c(G)$ and its image is dense in $\mathcal M_c(G)$. 
\end{theorem}

\begin{proof}
Let $F(\varphi)=\mu_{\varphi}$ for each $\varphi$ in $\mathcal C_c(G)$. As the support of $\lambda$ is $G$, hence $\supp \mu_{\varphi}=\supp \varphi$. Then, clearly, the mapping $F:\mathcal C_c(G)\to\mathcal M_c(G)$ is linear. On the other hand, if $(\varphi_{\alpha})_{\alpha\in A}$ converges to $\varphi$ in $\mathcal C_c(G)$, then let $L$ be a compact set in $G$ such that all the supports of the $\varphi_{\alpha}$'s are in $L$. As $\lim_{\alpha} \|\varphi_{\alpha}-\varphi\|_{\infty}\to 0$ we have for each $f$ in $\mathcal C(G)$
$$
|\langle F(\varphi_{\alpha}),f\rangle -\langle F(\varphi),f\rangle|\leq \int_L |\varphi_{\alpha}(x)-\varphi(x)| |f(x)|\,d\lambda(x)\leq
$$
$$
\|\varphi_{\alpha}-\varphi\|_{\infty} \int_L |f(x)|\,d\lambda(x)\to 0,
$$
which proves the continuity of $F$. Let $\varphi,\psi$ be in $\mathcal C_c(G)$, then we have for each $f$ in $\mathcal C(G)$
$$
\langle F(\varphi_{\alpha})*F(\psi_{\alpha}),f\rangle=\langle \mu_{\varphi}*\nu_{\varphi},f\rangle=\int_G f(x y) \varphi(x) \psi(y)\,d\lambda(x)\,d\lambda(y)= 
$$
$$
\int_G  f(z) \bigl[\int_G  \varphi(x) \psi(x^{-1} z)\,d\lambda(x)\bigr]\,d\lambda(z)= \langle \mu_{\varphi_{\alpha}*\psi_{\alpha}},f\rangle=\langle F(\varphi_{\alpha}*\psi_{\alpha}),f\rangle,
$$
hence $F$ is an algebra homomorphism. Finally, the last statement follows by regularization.
\end{proof}

\begin{corollary}\label{app2}
Suppose that $G$ is unimodular. Then the restriction of the mapping $\varphi\to \mu_{\varphi}$ to $\mathcal C_c(G//K)$ maps $\mathcal C_c(G//K)$ into $\mathcal M_c(G//K)$ and its image is dense in $\mathcal M_c(G//K)$.
\end{corollary}

\begin{proof}
We have to show only that if $\varphi$ is $K$-invariant, then $\mu_{\varphi}$ is $K$-invariant, too. Using the notation $F$ from the previous theorem we can write for each $f$ in $\mathcal C(G)$ and for each $k,l$ in $K$
$$
\langle \mu_{\varphi},f\rangle= \int_G f(x) \varphi(x)\,d\lambda(x)=\int_G f(k x l) \varphi(k x l)\,d\lambda(x)=\int_G f(k x l) \varphi(x)\,d\lambda(x).
$$
Integrating over $K\times K$ we have
$$
\langle \mu_{\varphi},f\rangle=\int_G \int_K \int_K f(k x l) \varphi(x)\,d\omega(k)\,d\omega(l)\,d\lambda(x)
$$
$$
=\int_G f^{\#}(x) \varphi(x)\,d\lambda(x)=\langle \mu_{\varphi},f^{\#}\rangle= \langle \mu_{\varphi}^{\#},f\rangle,
$$
that is, $\mu_{\varphi}^{\#}=\mu_{\varphi}$, hence $\mu_{\varphi}$ is $K$-invariant.
\vskip.1cm

Let $\mu^{\#}$ be in $\mathcal M_c(G//K)$ and $(\varphi_{\alpha})_{\alpha\in A}$ in $\mathcal C_c(G)$ such that 
$$
\lim_{\alpha} \mu_{\varphi_{\alpha}}=\mu.
$$
Then obviously
$$
\lim_{\alpha} \mu_{\varphi_{\alpha}}^{\#}=\lim_{\alpha} \mu_{\varphi_{\alpha}^{\#}}=\mu^{\#},
$$
and the proof is complete.
\end{proof}

\begin{theorem}\label{approx}
Let $G$ be unimodular. Then $(G,K)$ is a Gelfand pair if and only if the convolution algebra $\mathcal C_c(G//K)$ is commutative.
\end{theorem}

\begin{proof}
The necessity is obvious and the sufficiency is a consequence of the previous theorem.
\end{proof}

Using the fact that $\mathcal C^{\infty}$ functions form a dense subset in $\mathcal C_c(G)$ with respect to the sup norm, we have the following corollary.

\begin{corollary}\label{app3}
Suppose that $G$ is unimodular. Then the restriction of the mapping $\varphi\to \mu_{\varphi}$ to $\mathcal C_c(G//K)\cap \mathcal C^{\infty}(G)$ has a dense image in $\mathcal M_c(G//K)$.
\end{corollary}

For the proofs of the following results see e.g. \cite{MR2640609}.

\begin{theorem} (C.~Berg)
If $(G,K)$ is a Gelfand pair, then $G$ is unimodular.
\end{theorem}

\begin{theorem}\label{commG}
Suppose that there exists a continuous involutive automorphism $\theta:G\to G$ such that $\theta(x)$ is in $K x^{-1} K$ holds for each $x$ in $G$. Then $(G,K)$ is a Gelfand pair.
\end{theorem}

\begin{corollary}\label{Gelfand}
Suppose that there exists a continuous involutive automorphism $\theta:G\to G$ such that the subgroup of all $\theta$-fixed points $K$ is compact, and every element $x$ of $G$ can be written in the form $x=k y$ with $k$ in $K$ and $\theta(y)=y^{-1}$. Then $(G,K)$ is a Gelfand pair.
\end{corollary}

\begin{proof}
Let $x=k y$ with $k$ in $K$ and $\theta(y)=y^{-1}$, then 
$$
\theta(x)=\theta(k)\cdot \theta(y)=k\cdot y^{-1}=k\cdot x^{-1}\cdot k\in K x^{-1} K, 
$$
hence the statement follows from the previous theorem.
\end{proof}

\section{Semidirect products}\label{semi}
\hskip.5cm
Let $N$ be a locally compact group and $K$ a compact topological group of continuous automorphisms of $N$. Hence, as a group, $K$ is a subgroup of the group $\Aut(N)$ of all continuous automorphisms of $N$, and $K$ is equipped with a compact topology which is compatible with the group structure, that is, $k\mapsto k(n)$ is continuous on $K$ for each $n$ in $N$.  We consider the semidirect product $G=K\ltimes N$, where the operation is defined by
\begin{equation}\label{opdef}
(k,n)\cdot (l,m)=(k\cdot l, k(m)\cdot n)
\end{equation}
for each $(k,n),(l,m)$ in $G$. The identity of this group is $e=(id,e)$, where $id$ is the identity mapping on $N$ and $e$ is the identity element of $N$, further, the inverse of $(k,n)$ is
$$
(k,n)^{-1}=(k^{-1}, k^{-1}(n^{-1})).
$$
In general, $K$ is isomorphic to the compact subgroup  $\{(k,e):\, k\in K\}$ of $G$, and $N$ is isomorphic to the normal subgroup $\{id,n):\, n\in N\}$. 
\vskip.1cm

Suppose now that $N$ is commutative and we write the operation in $N$ as addition with $e=o$. If we define $\theta:G\to G$ by
$$
\theta(k,n)=(k,-n)
$$
for each $k$ in $K$ and $n$ in $N$, then we have that 
$$
\theta[(k,n) (l,m)]=\theta(k l,k(m)+n)=(k l, -k(m)-n)=
$$
$$
(k,-n)(l,-m)=\theta(k,n) \theta(l,m),
$$
that is, $\theta$ is a continuous involutive automorphism of $G$. On the other hand, we have
$$
\theta(k,n)=(k,-n)=(k,o) (k^{-1},k^{-1}(-n)) (k,o)=(k,o) (k,n)^{-1} (k,o),
$$
that is, $\theta(k,n)$ is in $K\cdot (k,n)^{-1} \cdot K$ for each $(k,n)$ in $G$. By Theorem \ref{commG}, we have the following result.

\begin{corollary}\label{semidir}
With the above notation $G=K\ltimes N$ we have that $(G,K)$ is a Gelfand pair.
\end{corollary}

In this case we can say -- somewhat loosely -- that $(N,K)$ is a Gelfand pair. The continuous function $f:K\times N\to\C$ is $K$-invariant if and only if
$$
f(k,n)=f\bigl((k',o)(k,n)(l',o)\bigr)=f\bigl(k' k l',k' (n)) 
$$
for each $k,k',l'$ in $K$ and $x$ in $\N$. With the choice $k'=l$ and $l'=k^{-1} l^{-1}$ it follows that $f(k,n)=f(id, l (n))$ for each $k,l$ in $K$ and $n$ in $N$. This means that $K$-invariant functions depend on the second variable only, that is, they can be identified with continuous functions on $N$ by restriction $f\mapsto f|_N$, and this restriction is invariant with respect to the action of $K$ on $N$: $f|_N(k(n))=f|_N(n)$ for each $k$ in $K$ and $n$ in $N$. These functions on $N$ are called {\it $K$-radial functions} and $\mathcal C(G//K)$ will be identified with the space $\mathcal C_K(N)$ of all $K$-radial functions. Hence $f$ is in $\mathcal C_K(N)$ if and only if $f:N\to\C$ is a continuous function satisfying $f(k\cdot n)=f(n)$ for each $k$ in $K$ and $n$ in $N$. The dual $\mathcal M_c(G//K)$ is the space of {\it $K$-radial measures} $\mu$ on $G$ satisfying
$$
\langle \mu,f\rangle = \int_G \int_K \int_K f(k' k l',k'(n))\,d\omega(k')\,d\omega(l')\,d\mu(k,n)
$$
for each continuous function $f:K\times N\to\C$. Clearly, $K$-radial measures can be identified with those measures $\mu$ on $N$, which satisfy 
$$
\int_N f(k(n))\,d\mu(n)=\int_N f(n)\,d\mu(n)
$$
for each continuous function $f:N\to\C$ and for every $k$ in $K$. The space of $K$-radial measures will be denoted by $\mathcal M_K(N)$.
\vskip.1cm

Given a continuous function $f:K\times N\to \C$ its $K$-projection is the function 
$$
f^{\#}(n)=\int_K f(k k' l', k(n))\,d\omega(k)\,d\omega(l')=\int_K f(k,k(n))\,d\omega(k).
$$
For each $(k,m)$ in $K\times N$ the $K$-radial measure $\delta_{(k,m)}^{\#}$ is independent of $k$: $\delta_{(k,m)}^{\#}=\delta_m^{\#}$, and for each $K$-radial function $f$ we have
$$
\tau_mf(n)=\delta_{-m}^{\#}*f(n)=\int_K f(n+k(m))\,d\omega(k)
$$

We obtain important special cases with the choice $N=\R^n$ and $K=O(n)$, or $K=SO(n)$, or $N=\C^n$ and $K=U(n)$, or $K=SU(n)$. In all these cases $(\R^n,K)$ is a Gelfand pair, and the space $\mathcal C(G//K)$, resp. $\mathcal M_c(G//K)$ will be identified with $\mathcal C_K(\R^n)$, resp. $\mathcal M_K(\R^n)$. Convolution in $\mathcal M_K(\R^n)$, resp. between $\mathcal M_K(\R^n)$ and $\mathcal C_K(\R^n)$ is the same as in $\mathcal M_c(\R^n)$, resp. between $\mathcal M_c(\R^n)$ and $\mathcal C(\R^n)$.

\section{Spherical functions}
\hskip.5cm
From now on we suppose that $(G,K)$ is a Gelfand pair. For every $f$ in $\mathcal C(G//K)$ and for each $y$ in $G$ the $K$-invariant measure
$$
D_{f;y}=\delta_{y^{-1}}^{\#}-f(y) \delta_e
$$
is called the {\it modified $K$-spherical difference} corresponding to $f$ and $y$ (see \cite{Sze14, Sze14b, Sze15}). Given $f$ in $\mathcal C(G//K)$ the closure of the $K$-ideal generated by all modified $K$-spherical differences of the form $D_{f;y}$ with $y$ in $G$ will be denoted by $M_f$.

\begin{theorem}\label{sphfunceq}
Let $f$ be in $\mathcal C(G//K)$. The $K$-ideal $M_f$ is proper if and only if $f(e)=1$, and $f$ satisfies
\begin{equation}\label{spherical}
\int_K f(x k y)\,d\omega(k)=f(x) f(y)
\end{equation}
for each $x,y$ in $G$. In this case $M_f$ is a maximal ideal and we have $\mathcal M_c(G//K)/M_f\cong \C$.
\end{theorem}

\begin{proof}
Suppose that $M_f$ is a proper ideal. Then $V=M_f^{\perp}$ is a nonzero $K$-variety, by Theorems \ref{anni1} and \ref{duality}. Let $\varphi\ne 0$ be in $V$, then we have for each $x,y,z$ in $G$
$$
0=\langle D_{f;y},\delta_{z}^{\#}*\varphi\rangle=\langle \delta_{y^{-1}}^{\#}, \delta_{z}^{\#}*\varphi\rangle-f(y)\, \delta_{z}^{\#}*\varphi(e)=
$$
$$
(\delta_{z}^{\#}*\varphi)^{\#}(y^{-1})-f(y)\, \int_K \int_K \varphi(k z^{-1} l)\,d\omega(k)\,d\omega(l)=
$$
$$
\int_K \int_K \int_K \varphi(z^{-1} l l' y^{-1} k')\,d\omega(l)\,d\omega(l')\,d\omega(k')-f(y)\, \int_K \int_K \varphi(k z^{-1} l)\,d\omega(k)\,d\omega(l)=
$$
$$
=\int_K \varphi(z^{-1} l y^{-1})\,d\omega(l) -f(y)\,\int_K \varphi(z^{-1} l)\,d\omega(l)=
$$
$$
\int_K \varphi(z^{-1} l y^{-1})\,d\omega(l)-f(y)\varphi(z^{-1}).
$$
The substitution $z=e$ gives $\varphi(y^{-1})=f(y)\varphi(e)$ thus $\varphi(e)\ne 0$ and we obtain equation \eqref{spherical} and $f(e)=1$.
\vskip.1cm

We have proved that $V$ consists of all constant multiples of $\widecheck{f}$. In particular, $V$ is a one dimensional vector space, which implies, by Theorem \ref{duality}, that $M_f$ is a maximal ideal. We show that $\mathcal M_c(G//K)/M_f$, as an algebra, is isomorphic to the algebra of complex numbers.
\vskip.1cm

For each $\mu$ in $\mathcal M_c(G//K)$ we define
$$
\Phi(\mu)=\langle\widecheck{\mu},f\rangle.
$$
Clearly, $\Phi:\mathcal M_c(G//K)\to\C$ is a surjective continuous linear functional. For $\mu,\nu$ in $\mathcal M_c(G//K)$ we have
$$
\Phi(\mu*\nu)=\langle \mu*\nu,\widecheck{f}\rangle=\int_G \int_G \widecheck{f}(u v)\,d\mu(u)\,d\nu(v)=
\int_G \int_G \widecheck{f}(u v)\,d\mu^{\#}(u)\,d\nu^{\#}(v)=
$$
$$
\int_G \int_G \int_K \int_K \int_K \int_K \widecheck{f}(k u l l' v k')\,d\omega(k)\,d\omega(l)\,d\omega(l')\,d\omega(k')\,d\mu(u)\,d\nu(v)=
$$
$$
\int_G \int_G \int_K \widecheck{f}(u l v)\,d\omega(l)\,d\mu(u)\,d\nu(v)= \int_G \int_G \int_K f(u l v)\,d\omega(l)\,d\widecheck{\mu}(u)\,d\widecheck{\nu}(v)=
$$
$$
\int_G \int_G f(u) f(v) \,d\widecheck{\mu}(u)\,d\widecheck{\nu}(v)=\Phi(\mu) \Phi(\nu),
$$
that is, $\Phi$ is an algebra homomorphism. On the other hand,
$$
\Phi(D_{f;y})=\Phi(\delta_{y^{-1}}^{\#})-f(y)=
$$
$$
\langle (\delta_{y^{-1}}^{\#})\,\widecheck{},f\rangle-f(y)= \langle \delta_{y}^{\#},f\rangle-f(y)=f^{\#}(y)-f(y)=0,
$$
as $f$ is $K$-invariant. It follows that $M_f$ is a subset of the kernel of $\Phi$. As $\Phi$ is continuous, its kernel is a closed maximal ideal, hence, in fact, $M_f$ is the kernel of $\Phi$. As $\Phi$ is surjective, we have $\mathcal M_c(G//K)/M_f\cong \C$.
\vskip.1cm

Now we suppose that $f(e)=1$, and $f$ satisfies \eqref{spherical}. Then it is easy to check that $\widecheck{f}$ is in $M_f^{\perp}$, hence $M_f$ is proper, by Theorem \ref{duality}. The theorem is proved.
\end{proof}

We call the nonzero $K$-invariant function $f$ a {\it $K$-spherical function}, if it satisfies equation \eqref{spherical} for each $x,y$ in $G$. In this case $f(e)=1$. By the previous theorem, there is a one-to-one correspondence between $K$-spherical functions on $G$ and those maximal ideals of the algebra $\mathcal M_c(G//K)$ whose residual algebra is topologically isomorphic to $\C$. Such maximal ideals are called -- in accordance with the terminology in the commutative case -- {\it exponential maximal ideals} (see e.g. \cite{Sze14}).
In other words, a maximal ideal in $\mathcal M_c(G//K)$ is exponential, if it is the kernel of a continuous algebra homomorphism of $\mathcal M_c(G//K)$ onto $\C$. In particular, every exponential maximal ideal is closed. It follows immediately that, in the case of commutative $G$, the $K$-spherical functions are exactly the exponentials of the group $G/K$.
\vskip.1cm

We have the following characterization of $K$-spherical functions. We call the function $f$ in $\mathcal C(G)$ {\it normed} if $f(e)=1$.

\begin{theorem}\label{sphchar}
The $f$ be a continuous $K$-invariant function. Then the following statements are equivalent:
\begin{enumerate}[i)]
\item The function $f$ is a $K$-spherical function.
\item The function $f$ is nonzero and satisfies \eqref{spherical} for each $x,y$ in $G$.
\item The function $f$ is normed and for each $K$-invariant measure $\mu$ there exists a complex number $\lambda_{\mu}$ such that $\mu*f=\lambda_{\mu}\cdot f$
\item The function $f$ is a common normed eigenfunction of all translation operators $\tau_y$ with $y$ in $G$.
\item The function $f$ is normed and the ideal $M_f$ is an exponential maximal ideal.
\item The function $f$ is normed and the mapping $\mu\mapsto \langle \mu,\widecheck{f}\rangle$ is a nonzero multiplicative functional of the algebra $\mathcal M_c(G//K)$ with kernel $M_f$.
\end{enumerate}
\end{theorem}

\begin{proof}
The first two statements are equivalent, by definition.
\vskip.1cm

Suppose that $f$ is nonzero and satisfies \eqref{spherical} for each $x,y$ in $G$.
Let $\mu$ be a $K$-invariant measure, then we have
$$
\mu*f(x)=\int_G f(y^{-1} x)\,d\mu(y)=\int_G \widecheck{f}(x^{-1} y)\,d\mu(y)= \int_G \widecheck{f}(x^{-1} y)\,d\mu^{\#}(y)=
$$
$$
\int_G \int_K \int_K \widecheck{f}(x^{-1} k y l)\,d\omega(k)\,d\omega(l)\,d\mu(y)= \int_G \int_K  \widecheck{f}(x^{-1} k y)\,d\omega(k)\,d\mu(y)=
$$
$$
\int_G \int_K  f(y^{-1} k x)\,d\omega(k)\,d\mu(y)=\int_G  f(y^{-1}) \,d\mu(y)\cdot f(x),
$$
which proves $iii)$ with $\lambda_{\mu}=\int_G  f(y^{-1}) \,d\mu(y)$.
\vskip.1cm

As $\tau_yf=\delta_{y^{-1}}*f$, $iii)$ obviously implies $i)$.
\vskip.1cm

If $f$ is a common normed eigenfunction of all translation operators $\tau_y$ with $y$ in $G$, then $\tau(f)$ is a one dimensional variety. It is easy to see that  $\tau(\widecheck{f})$ is a one dimensional variety, too. Hence $\tau(\widecheck{f})^{\perp}$ is a maximal ideal in $\mathcal M_c(G//K)$, by Theorem \ref{duality}. In the proof of Theorem \ref{sphfunceq} we have seen that $\tau(\widecheck{f})^{\perp}=M_f$. The proof of the statement that $M_f$ is exponential is included in the proof of Theorem \ref{sphfunceq}, too.
\vskip.1cm

Now we suppose that the ideal $M_f$ is an exponential maximal ideal and $f$ is normed. Then we define $\Phi(\mu)=\langle \mu,\widecheck{f}\rangle$ for each $\mu$ in $\mathcal M_c(G//K)$. We can perform the same calculation as in the proof of Theorem \ref{sphfunceq} to show that $\Phi$ is a multiplicative functional of the algebra $\mathcal M_c(G//K)$. As 
$$
\Phi(\delta_{e}^{\#})=\langle \delta_e,\widecheck{f}\rangle=\widecheck{f}(e)=f(e)=1,
$$
hence $\Phi$ is nonzero. The statement about the kernel of $\Phi$ is proved in Theorem \ref{sphfunceq}, too.
\vskip.1cm

Finally, we suppose that the mapping $\mu\mapsto \langle \mu,\widecheck{f}\rangle$ is a nonzero multiplicative functional of the algebra $\mathcal M_c(G//K)$ with kernel $M_f$ and $f$ is normed. We have for each $x$ in $G$
$$
f(x)=\langle\delta_{x^{-1}}^{\#},\widecheck{f}\rangle=\Phi(\delta_{x^{-1}}^{\#}),
$$
hence
$$
f(x) f(y)=\Phi(\delta_{x^{-1}}^{\#}) \Phi(\delta_{y^{-1}}^{\#})=\Phi(\delta_{x^{-1}}^{\#}*\delta_{y^{-1}}^{\#})=
\langle \delta_{x^{-1}}^{\#}*\delta_{y^{-1}}^{\#},\widecheck{f}\rangle=
$$
$$
\int_G \int_G \widecheck{f}(u v)\,d\delta_{x^{-1}}^{\#}(u)\,d\delta_{y^{-1}}^{\#}(v)=
$$
$$
\int_G \int_G \int_K \int_K \int_K \int_K \widecheck{f}(k u l k' v l')\,d\omega(k)\,\,d\omega(l)\,d\omega(k')\,d\omega(l')\,d\delta_{x^{-1}}(u)\,d\delta_{y^{-1}}(v)=
$$
$$
\int_G \int_G \int_K \widecheck{f}(u l  v ) \,d\omega(l)\,d\delta_{x^{-1}}(u)\,d\delta_{y^{-1}}(v)=
$$
$$
\int_K \widecheck{f}(x^{-1} l  y^{-1} ) \,d\omega(l)=\int_K f(y l x) \,d\omega(l),
$$
and the theorem is proved.
\end{proof}

\section{The case $N=\R^n$ and $K=SO(n)$}\label{spec}
\hskip.5cm
We consider the general situation exhibited in Section \ref{semi} with $N=\R^n$ and $K=SO(n)$, the {\it special orthogonal group}. In this case the semidirect product $G=SO(n)\ltimes \R^n$ is called the {\it group of Euclidean motions} (see also \cite{MR2640609}). The elements $(k,a)$ in $G$ can be thought as the product of a rotation $k$ in $SO(n)$ and a translation by $a$ in $\R^n$. Hence the pair $g=(k,a)$ operates on $\R^n$ by the rule
$$
g\cdot x=k\cdot x+a
$$
for each $x$ in $\R^n$. By Corollary \ref{semidir}, we conclude that $(\R^n,SO(n))$ is a Gelfand pair. $K$-radial functions are those continuous functions $f:\R^n\to\C$ satisfying
$$
f(x)=f(k\cdot x)
$$
whenever $x$ is in $\R^n$ and $k$ is a real orthogonal $n\times n$ matrix with \hbox{determinant $+1$.} Similarly, the compactly supported measure $\mu$ is $K$-radial if and only if it satisfies
$$
\int_{\R^n} f(x)\,d\mu(x)=\int_{\R^n} f(k\cdot x)\,d\mu(x)
$$
for each continuous function $f:\R^n\to\C$ and for each real orthogonal $n\times n$ matrix $k$ with \hbox{determinant $+1$.}
\vskip.1cm

The proof of the following result can be found in \cite{MR2640609}.

\begin{theorem}\label{Lap}
The $K$-radial function $\varphi:\R^n\to\C$ is a $K$-spherical function if and only if it is $\mathcal C^{\infty}$, it is an eigenfunction of the Laplacian, and $\varphi(0)=1$.
\end{theorem}

Let $\varphi$ be a $\mathcal C^{\infty}$ $K$-radial function on $\R^n$, which is a solution of $\Delta \varphi=\lambda \varphi$ for some complex number $\lambda$. Let $\varphi_0$ be defined for real $r$ as
$$
\varphi_0(r)=\varphi(r,0,0,\dots,0),
$$
then
$$
\varphi(x)=\varphi_0(\|x\|)
$$
holds for each $x$ in $\R^n$ and $\varphi_0$ is a regular even solution of the differential equation 
$$
\frac{d^2 \varphi_0}{dr^2}+\frac{n-1}{r} \frac{d\varphi_0}{dr}=\lambda \varphi_0,
$$
hence it is proportional to the Bessel--function $J_{\lambda}$ defined by
$$
J_{\lambda}(r)=\Gamma\bigl(\frac{n}{2}\bigr)\,\sum_{k=0}^{\infty}\,\frac{\lambda^k}{k!\,\Gamma(k+\frac{n}{2})}\,\Bigl(\frac{r}{2}\Bigr)^{2k}.
$$
As $J_{\lambda}(0)=1$ follows that $\varphi$ is a $K$-spherical function if and only if 
$$
\varphi(x)=J_{\lambda}(\|x\|)
$$
holds for each $x$ in $\R^n$ with some complex number $\lambda$ (see \cite{MR2640609}).

\section{Spherical monomials}
\hskip.5cm 
The function $\varphi$ in $\mathcal C(G//K)$ is called a {\it generalized $K$-spherical monomial}, if there exists a $K$-spherical function $s$ and a natural number $n$ such that we have for each $x,y_1,y_2,\dots,y_{n+1}$
\begin{equation}\label{diffeq}
D_{s;y_1}*D_{s;y_2}*\dots*D_{s,y_{n+1}}*\varphi(x)=0.
\end{equation}

\begin{lemma}
The nonzero function $\varphi$ in $\mathcal C(G//K)$ is a generalized $K$-spherical monomial if and only if there exists a unique exponential maximal ideal in $\mathcal M_c(G//K)$ and a natural number $n$ such that
$$
M^{n+1}\subseteq \Ann \tau(\varphi).
$$
\end{lemma}

\begin{proof}
Let $\varphi\ne 0$ be a generalized $K$-spherical monomial. Then there exists a $K$-spherical function $s$ and a natural number $n$ such that \eqref{diffeq} holds for each $x,y_1,y_2,\dots,y_{n+1}$ in $G$, and, by the commutativity of $\mathcal M_c(G//K)$, we have
$$
D_{s;y_1}*D_{s;y_2}*\dots*D_{s,y_{n+1}}*\psi(x)=0
$$
whenever $\psi$ is in $\tau(\varphi)$. As the measures $D_{s;y_1}*D_{s;y_2}*\dots*D_{s,y_{n+1}}$ generate the ideal whose closure is $M_s^{n+1}$ we infer that $M_s^{n+1}\subseteq \Ann \tau(\varphi)$. As $\varphi$ is nonzero, there exists a maximal ideal $M$ in $\mathcal M_c(G//K)$ such that $\Ann \tau(\varphi)\subseteq M$, which implies 
$$
M_s^{n+1}\subseteq M.
$$
Maximal ideals are prime, hence we conclude $M=M_s$, which is an exponential maximal ideal. If $N$ is a maximal $K$-ideal with the property that
$$
N^{k+1}\subseteq \Ann \tau(\varphi)
$$
for some natural number $n$, then we have 
$$
N^{k+1}\subseteq M_s,
$$
and $M_s$ is prime, hence we conclude $N=M_s$. 
\vskip.1cm

The converse statement is obvious.
\end{proof}

If $s$ is a $K$-spherical function and $M_s^{n+1}\subseteq \Ann \tau(\varphi)\subseteq M_s$ holds for some natural number $n$, then we say that the generalized $K$-spherical monomial $\varphi$ {\it corresponds to $s$}. By the above lemma, $s$ is unique. The smallest natural number $n$ with this property is called the {\it degree} of $\varphi$. The zero function is a generalized $K$-spherical monomial but we do not assign a degree to it. 
\vskip.1cm

We call a generalized $K$-spherical monomial simply a {\it $K$-spherical monomial}, if it generates a finite dimensional $K$-variety. 
By the definition, the set of all generalized $K$-spherical monomials of degree $n$ corresponding to the $K$-spherical function $s$ is $\Ann M_s^{n+1}$ and the set of all generalized $K$-spherical monomials corresponding to $s$ is $\bigcup_{n\in\N} \Ann M_s^{n+1}$. Further, given the $K$-variety $V$ the set of all generalized $K$-spherical monomials of degree $n$ corresponding to the $K$-spherical function $s$ in $V$ is $V\cap \Ann M_s^{n+1}$ and the set of all generalized $K$-spherical monomials corresponding to $s$ in $V$ is $V\cap \bigcup_{n\in\N} \Ann M_s^{n+1}$.

\section{Spherical spectral analysis and spectral\\ synthesis}
\hskip.5cm
Let $V$ be a $K$-variety. We say that {\it $K$-spectral analysis} holds for $V$, if in every nonzero sub-$K$-variety of $V$ there is a $K$-spherical function. If $G$ is commutative, then this is equivalent to spectral analysis for the variety $V$ in $\mathcal C(G/K)$.

\begin{theorem}
Let $V$ be a $K$-variety. $K$-spectral analysis holds for $V$ if and only if every maximal ideal in $\mathcal M_c(G//K)$ containing $\Ann V$ is exponential. In other words, $K$-spectral analysis holds for $V$ if and only if every maximal ideal in $\mathcal M_c(G//K)/\Ann V$ is exponential.
\end{theorem} 

\begin{proof}
If $K$-spectral analysis holds for $V$ and $M$ is a maximal $K$-ideal with $\Ann V\subseteq M$, then obviously $V=\Ann \Ann V\supseteq \Ann M$. Hence $\Ann M$ is a nonzero sub-$K$-variety of $V$, which includes a $K$-spherical function $s$. It follows that $\Ann \tau(s)$, which is a maximal $K$-ideal, is a superset of $M$, hence they are equal:  $\Ann \tau(s)=M$. As $\Ann \tau(s)$ is exponential, the necessity part of theorem is proved.
\vskip.1cm

Suppose now that every maximal $K$-ideal containing $\Ann V$ is exponential and let $W$ be a nonzero sub-$K$-variety in $V$. Then $\Ann W\supseteq \Ann V$, hence every maximal $K$-ideal containing $\Ann W$ also contains $\Ann V$, thus it is exponential. Let $M$ be one of them; then we have $M=M_s$ for some $K$-spherical function $s$, further $\Ann W\supseteq M_s=\Ann \tau(s)$. We conclude $\tau(s)\subseteq W$, and the theorem is proved.
\end{proof}

This is the analogue of the spectral analysis theorems Theorem 14.2 and Theorem 14.3 on p.~203 in \cite{Sze14}.
\vskip.3cm

We say that {\it $K$-spectral analysis holds on $G$}, if $K$-spectral analysis holds for each $K$-variety on $G$. This means that $K$-spectral analysis holds for $\mathcal C(G//K)$. If $G$ is commutative, then this is exactly spectral analysis on the group $G/K$.

\begin{corollary}
$K$-spectral analysis holds on $G$ if and only if every maximal ideal in $\mathcal M_c(G//K)$ is exponential.
\end{corollary}

For instance, if $G$ is a discrete Abelian group and $K$ is any finite subgroup, then $\mathcal M_c(G//K)$ is isomorphic to $\mathcal M_c(G/K)$. In this case the condition of the theorem is satisfied if and only if the torsion-free rank of the group $G/K$ is less than the continuum (see e.g. \cite{LSz04}).
\vskip.1cm

Let $V$ be a $K$-variety. We say that $V$ is {\it $K$-synthesizable}, if the $K$-spherical monomials span a dense subspace in $V$. We say that {\it $K$-spectral synthesis} holds for $V$, if every sub-$K$-variety of $V$ is $K$-synthesizable. We say that {\it $K$-spectral synthesis} holds on $G$, if every $K$-variety on $G$ is $K$-synthesizable. It is easy to see, that $K$-spectral synthesis implies $K$-spectral analysis for a variety. Clearly, $K$-synthesizability and $K$-spectral synthesis reduce to synthesizability and spectral synthesis on $G/K$ if $G$ is commutative. If, for instance, $G$ is a discrete Abelian group, and $K$ is a finite subgroup, then $K$-spectral synthesis holds on $G$ if and only if the torsion-free rank of $G/K$ is finite (see \cite{LSz06}).
\vskip.1cm

For synthesizability of varieties we have the following result (see \cite{Sze14, Sze14b, Sze15}).

\begin{theorem}\label{synthvar}
The nonzero $K$-variety $V$ is $K$-synthesizable if and only if 
$$
\Ann V=\bigcap_{M}\,\bigcap_{n\in\N} (\Ann V+M^{n+1}),
$$
where the first intersection is taken for all exponential maximal ideals $M$ containing $\Ann V$ and $\mathcal M_c(G//K)/M^{n+1}$ is finite dimensional.
\end{theorem}

\begin{proof}
By definition of synthesizability and by the remarks at the end of the previous section, the $K$-variety $V$ is synthesizable if and only if
$$
V=\sum_M\,\sum_{n\in\N} (V\cap \Ann M^{n+1}).
$$
On the other hand, applying Theorem 8. on p.~6 in \cite{Sze15}, our statement follows.
\end{proof}

In the case if $N$ is a locally compact Abelian group and $K$ is a compact group of its automorphisms as it has been discussed in Section \ref{semi} instead of "$K$-spectral analysis on $G$", resp. "$K$-spectral synthesis on $G$" we can say simply "$K$-spectral analysis on $N$", resp. "$K$-spectral synthesis on $N$".

\section{Extension of L.~Schwartz's spectral synthesis}
\hskip.5cm
In his monumental and pioneer work \cite{Sch47} L.~Schwartz proved the following fundamental result.

\begin{theorem}\label{sp} (L.~Schwartz)
Spectral synthesis holds on the reals. In other words, in every linear and translation invariant space of continuous complex valued functions on the real line, which is closed with respect to uniform convergence on compact sets all functions of the form $x\mapsto x^n e^{\lambda x}$ span a dense subspace ($n$ is a natural number, $\lambda$ is a complex number).  
\end{theorem}

As a consequence, every complex valued continuous function on the real line can be uniformly approximated on compact sets by linear combinations of functions of the above form, the {\it exponential monomials}, which are uniform limits on compact sets of linear combinations of translates of the given function. This statement clearly includes the deep existence theorem on spectral analysis: given any nonzero continuous complex valued function on the reals the smallest linear space including the function and is closed under translation and uniform convergence on compact sets contains nonzero exponential monomials. The proof of this beautiful result depends on hard complex function theory. Since the above result was published a great number of efforts have been made to extend it to several variables, that is, to varieties in $\R^n$, but an extension was possible only in the case of some special varieties, no general result was available. Finally, almost thirty years later D.~I.~Gurevich provided counterexamples for the corresponding result in $\R^2$. In fact, in \cite{MR0390759} he gave an example for a variety in $\R^2$ whose annihilator is generated by two measures and no spectral analysis holds on it, and for another variety whose annihilator is generated by three measures and no spectral synthesis holds on it. These negative results verify the conjecture that a direct generalization to translation invariant closed subspaces in $\R^n$ may not be proper way to extend Schwartz's result. In fact, here we show that a more sophisticated, but still natural generalization is possible in terms of spherical functions. Our basic observation is the following: instead of spectral synthesis we consider $K$-spectral synthesis on $\R^n$, where $K=SO(n)$, the special orthogonal group acting on $\R^n$ as it was discussed in Section \ref{spec}. In the special case $n=1$ we have $K=SO(1)=\{id\}$, hence in this case $K$-spectral synthesis reduces to ordinary spectral synthesis. However, if $n>1$, then $K=SO(n)$ is non-trivial, and $K$-spectral synthesis is essentially different from ordinary spectral synthesis. Hence our forthcoming result shows that a successful generalization of L.~Schwartz's result in this direction is possible if "translation invariance" is replaced by "invariance with respect to Euclidean motions". To prove this statement we shall need some preliminary results. In what follows we let $K=SO(n)$, the special orthogonal group.
\vskip.1cm

For each $\mu$ in $\mathcal M_c(\R)$ and $f$ in $\mathcal C_K(\R^n)$ we define $\mu_K$ by the equation
$$
\langle \mu_K, f\rangle = \langle \mu, f_0\rangle,
$$
where the function $f_0:\R\to\C$ is given by
$$
f_0(r)=f(r,0,0,\dots,0)
$$
for each $r$ in $\R$. As $f$ is $K$-radial, and $K=SO(n)$ acts transitively on the unit sphere in $\R^n$ we have that $\|x\|=\|y\|$ implies $f(x)=f(y)$ for $x,y$ in $\R^n$. In particular, $f(\|x\|,0,0,\dots,0)=f(x)$, and
$$
f(x)=f_0(\|x\|)
$$
holds for each $x$ in $\R^n$. 

\begin{theorem}\label{epim}
The mapping $\mu\mapsto \mu_K$ is a continuous algebra homomorphism of $\mathcal M_c(\R)$ onto $\mathcal M_K(\R^n)$.
\end{theorem}

\begin{proof}
It is easy to see that $\mu_K$ is a compactly supported measure on $\R^n$, hence it is a linear functional on $\mathcal C_K(\R^n)$. Further, we have for each  $\mu,\nu$ in $\mathcal M_c(\R)$ and $f$ in $\mathcal C_K(\R^n)$
$$
\langle (\mu*\nu)_K,f\rangle =\langle \mu*\nu,f_0\rangle=\int_{\R} \int_{\R} f_0(s+r)\,d\mu(s)\,d\nu(r)=
$$
$$
\int_{\R} \int_{\R} f(s+r,0,0,\dots,0)\,d\mu(s)\,d\nu(r),
$$
and
$$
\langle \mu_K*\nu_K,f\rangle=\int_{\R^n} \int_{\R^n} f(x+y)\,d\mu_K(x)\,d\nu_K(y)=
$$
$$
\int_{\R^n} \Bigl[\int_{\R} f(s+y_1,y_2,\dots,y_n)\,d\mu(s)\Bigr]\,d\nu_K(y)=
$$
$$
\int_{\R} \Bigl[\int_{\R^n} f(s+y_1,y_2,\dots,y_n)\,d\nu_K(y)\Bigr]\,d\mu(s)= 
\int_{\R} \int_{\R} f(s+r,0,0,\dots,0)\,d\nu(r)\,d\mu(s),
$$
\vskip.1cm
\noindent hence $ (\mu*\nu)_K=\mu_K*\nu_K$, and $\mu\mapsto \mu_K$ is an algebra homomorphism. Finally, for each $\varphi$ in $\mathcal M_c(\R)$ and $x$ in $\R^n$ we define
$$
f_{\varphi}(x)=\varphi(\|x\|).
$$
Then $f_{\varphi}$ is in $\mathcal C_K(\R^n)$. Given $\xi$ in $\mathcal M_K(\R^n)$ we let
$$
\langle \mu,\varphi\rangle=\langle \xi,f_{\varphi}\rangle.
$$
We then have for each $f$ in $\mathcal C_K(\R^n)$
$$
\langle \mu_K,f\rangle=\langle \mu,f_0\rangle=\langle \xi,f_{f_0}\rangle=\int_{\R^n} f_{f_0}(x)\,d\xi(x)=\int_{\R^n} f_0(\|x\|)\,d\xi(x)=
$$
$$
\int_{\R^n} f(\|x\|,0,0,\dots,0)\,d\xi(x)=\int_{\R^n} f(x)\,d\xi(x)=\langle \xi,f\rangle,
$$
which means $\mu_K=\xi$, hence the mapping $\mu\mapsto \mu_K$ is surjective. Its continuity is obvious.
\end{proof}

By Theorem \ref{synthvar}, the $K$-synthesizability of a $K$-variety can be expressed purely in terms of the annihilator of the $K$-variety. We introduce the following terminology: let $R$ be a commutative complex topological algebra with unit. The proper closed ideal $I$ in $R$ is called {\it synthesizable} if 
\begin{equation}\label{synth}
I=\bigcap_{M}\,\bigcap_{n\in\N} (I+M^{n+1}),
\end{equation}
where the first intersection is taken for all exponential maximal ideals $M$ containing $I$ and $R/M^{n+1}$ is finite dimensional. Accordingly, we say that spectral synthesis holds on $R$, if every closed ideal $I$ in $R$ satisfies the above equation. In particular, $K$-spectral synthesis holds on $G$ if and only if this equation holds for each proper closed ideal $I$ in $\mathcal M_c(G//K)$. The following theorem is a simple consequence.

\begin{theorem}\label{homo}
Let $R,Q$ be commutative complex topological algebras with unit. If spectral synthesis holds  on $R$, and there exists a continuous surjective homomorphism \hbox{$\Phi:R\to Q$,} then spectral synthesis holds on $Q$.
\end{theorem}

\begin{proof}
Let $M$ be a maximal ideal with $Q$, then $M=\Phi(N)$ with some ideal $N$ in $R$ such that $N=\Phi^{-1}(M)$. Let $\psi:Q\to Q/M$ denote the natural mapping, then $\psi$ is continuous and open. We define
$$
F(r)=\psi\bigl(\Phi(r)\bigr)
$$
for each $r$ in $R$, then $F:R\to Q/M$ is a continuous homomorphism. Clearly, $F$ is surjective. If $F(r)=0$, then $\Phi(r)$ is in $\Ker \psi=M$, that is, $r$ is in $N$. It follows that $R/N\cong Q/M$, a field, hence $N=\Ker F$ is a closed maximal ideal. By assumption, $N$ is exponential, hence $M$ is exponential, too.
\vskip.1cm

Let $J$ be a proper closed ideal in $Q$ and let $I=\Phi^{-1}(J)$. Then $I$ is a proper closed ideal in $R$, hence it is synthesizable, by assumption. It follows that \eqref{synth} holds. Then we have 
\begin{equation}\label{synth2}
J=\bigcap_{\Phi(M)}\,\bigcap_{n\in\N} (J+\Phi(M)^{n+1}),
\end{equation}
and here the first intersection extends for all maximal ideals $\Phi(M)$ \hbox{containing $J$.} Indeed, the left hand side is clearly a subset of the right hand side. Suppose now that $q=\Phi(r)$ is not in $J$, then $r$ is not in $I$. By equation \eqref{synth}, there exists a maximal ideal $M$ with $I\subseteq M$, and a natural number $n_0$ such that $r$ is not in $I+M^{n_0+1}$, hence $q=\Phi(r)$ is not in $J+\Phi(M)^{n_0+1}$. It follows that \eqref{synth2} holds.
\vskip.1cm

What is left is to show that $Q/\Phi(M)^{n+1}$ is finite dimensional for every maximal ideal $M$ with $I\subseteq M$ and for each natural number $n$. We define $F:R/M^{n+1}\to Q/\Phi(M)^{n+1}$ by
$$
F(r+M^{n+1})=\Phi(r)+\Phi(M)^{n+1}
$$
for each $r$ in $R$. We have to show that the value of $F$ is independent of the choice of $r$ in the coset $r+M^{n+1}$. Suppose that $r-r_1$ is in $M^{n+1}$, that is $r-r_1=\sum x_1 x_2 \cdots x_{n+1}$, where the sum is finite, and $x_1,x_2,\dots,x_{n+1}$ is in $M$. Then 
$$
\Phi(r)=\Phi(r_1)+\sum \Phi(x_1) \Phi(x_2)\cdots \Phi(x_{n+1}),
$$
hence $\Phi(r)$ and $\Phi(r_1)$ are in the same coset of $\Phi(M)^{n+1}$. As $F$ is clearly a surjective homomorphism, we infer that $Q/\Phi(M)^{n+1}$ is finite dimensional and the proof is complete.
\end{proof}

Now we are ready to present our main result on the extension of  L.~Schwartz's spectral synthesis result Theorem \ref{sp}.

\begin{theorem}
Let $K=SO(n)$ for every positive integer $n$ acting on $\R^n$. Then $K$-spectral synthesis holds on $\R^n$.
\end{theorem}

\begin{proof}
Our statement is a consequence of L.~Schwartz's Theorem \ref{sp} using Theorem \ref{epim}, and Theorem \ref{homo}.
\end{proof}

\end{document}